\documentclass[10pt]{amsart}
\usepackage{amssymb}
\usepackage{amsmath}
\usepackage{amscd}
\usepackage[all]{xy}
\usepackage[colorlinks=true, linkcolor=red,
urlcolor=blue]{hyperref}
\usepackage{tikz}
\usepackage{pgfplots}
\pgfplotsset{compat = newest}
\newcounter{TmpEnumi}
\numberwithin{equation}{section}

\def\today{\number\day\space\ifcase\month\or
 January\or February\or
   March\or April\or May\or June\or
    July\or August\or September\or
   October\or November\or December\fi\
     \number\year}
\theoremstyle{definition}
\newtheorem{thm}{Theorem}[section]
\newtheorem{lem}[thm]{Lemma}
\newtheorem{prp}[thm]{Proposition}
\newtheorem{dfn}[thm]{Definition}

\newtheorem{rmk}[thm]{Remark}
\newtheorem{ntn}[thm]{Notation}
\newtheorem{exa}[thm]{Example}

\newcommand{\beq}{\begin{equation}}
\newcommand{\eeq}{\end{equation}}
\newcommand{\beqr}{\begin{eqnarray*}}
\newcommand{\eeqr}{\end{eqnarray*}}
\newcommand{\bal}{\begin{align*}}
\newcommand{\eal}{\end{align*}}
\newcommand{\bei}{\begin{itemize}}
\newcommand{\eei}{\end{itemize}}
\newcommand{\limi}[1]{\lim_{{#1} \to \infty}}

\newcommand{\af}{\alpha}

\newcommand{\dt}{\delta}
\newcommand{\ep}{\varepsilon}

\newcommand{\ld}{\lambda}

\newcommand{\C}{{\mathbb{C}}}
\newcommand{\N}{{\mathbb{Z}}_{> 0}}
\newcommand{\Nz}{{\mathbb{Z}}_{\geq 0}}

\newcommand{\spec}{{\operatorname{sp}}}

\newcommand{\card}{{\operatorname{card}}}
\newcommand{\Aut}{{\operatorname{Aut}}}

\newcommand{\cK}{{\mathcal{K}}}

\newcommand{\andeqn}{\qquad {\mbox{and}} \qquad}

\newcommand{\Wolog}{Without loss of generality}

\newcommand{\tfae}{the following are equivalent}

\newcommand{\ca}{C*-algebra}

\newcommand{\CGAa}{C^* (G, A, \af)}

\newcommand{\cfn}{continuous function}

\renewcommand{\S}{\subset}

\title[Weak tracial approximate representability]{
Weakly tracially approximately representable actions}  

\author{M. Ali Asadi-Vasfi}

\curraddr{Department of Mathematics, University of Toronto, Toronto, Ontario, M5S~2E4,  Canada.}

\date{\today}

\begin{document}
\maketitle
\begin{abstract}
We describe a weak tracial analog  of approximate representability under the name ``weak tracial approximate representability'' 
for finite group actions.
 We then investigate the dual actions on the crossed products by this class of group actions.
 Namely, let $G$ be a finite abelian group,
let $A$ be an  infinite-dimensional simple unital C*-algebra,
and let $\alpha \colon G \to \Aut (A)$ be an action of $G$ on $A$ which is pointwise outer.
Then $\alpha$ has the weak tracial Rokhlin property 
if and only if 
the dual action $\widehat{\alpha}$ of the  Pontryagin dual $\widehat{G}$ on the crossed product $C^*(G, A, \alpha)$
is weakly tracially approximately representable, and $\alpha$ is weakly tracially approximately representable 
if and only if 
 the dual action $\widehat{\alpha}$ has the weak tracial Rokhlin property.
This generalizes the results of Izumi in 2004 and Phillips in 2011 on the dual actions of finite abelian groups on unital
simple C*-algebras.
\end{abstract}
\tableofcontents
\section{Introduction}
Let $A$ be any C*-algebra, let $G$ be a locally compact abelian group, let 
$\alpha \colon G \to \Aut(A)$ be an action, and let $\widehat{G}$ be 
the Pontryagin dual of $G$.  
For every $\omega \in \widehat{G}$, there is an automorphism 
$\widehat{\alpha}_{\omega}$ of the crossed product $C^*(G, A, \alpha)$ given on $C_{c} (G, A, \alpha)$ by
$\widehat{\alpha}_{\omega} (a)(g) = \overline{\omega(g)} a(g)$ for $a \in C_{c} (G, A, \alpha)$ and $g \in G$. 
Moreover, 
$\widehat{\alpha} \colon \widehat{G} \to \Aut \left(C^*(G, A, \alpha)\right)$ 
is a continuous action of $\widehat{G}$ on $C^* (G, A, \alpha)$ 
which is called the dual action. (One can also use $\omega(g)$ in place of $\overline{\omega(g)}$.)
 In this regard, discovering structures and properties of the dual action $\widehat{\alpha}$  on the crossed product 
from structures and properties of the underlying action $\alpha$ or the other way around is always an important problem,
given the fact that $\widehat{\widehat{\alpha}}$ is essentially as same as $\alpha$ by the Takai duality.

The notion of approximate representability for finite group actions on C*-algebras was introduced by Izumi in \cite{Iz1}. He proved that 
an action $\alpha \colon G \to \Aut (A)$ of a finite abelian group on a unital separable C*-algebra has the Rokhlin
property if and only if the dual action $\widehat{\alpha}$ is approximately representable, and $\alpha$ is approximately representable if and only if the dual the dual action $\widehat{\alpha}$ has the Rokhlin property. 
This theorem is usually called 
the duality theorem for finite group actions  with the Rokhlin property.

In the setting of finite quantum group actions, Kodaka–Teruya 
studied the Rokhlin property and approximate representability in \cite{KoTe}.
Further, an appropriate notion of the Rokhlin property under the name 
``spatial Rokhlin property'' and appropriate version of approximate representability under the name
 ``spatial approximate representability'' were introduced in the setting of compact quantum groups by Gardella,  Kalantar, and Lupini in  \cite{GKL19}. The work of Izumi (the duality theorem for finite group actions) was generalized to these classes of quantum groups as well.
 Further, a version of the duality theorem for Rokhlin dimension, introduced by Hirshberg,  Winter,  and Zacharias in \cite{HWZ15},
  was obtained in \cite{GHS21} by Gardella,  Hirshberg, and  Santiago in the setting of representability dimensions and Rokhlin dimensions.
 
Even though a version of the tracial Rokhlin property for actions of non-finite compact groups has been recently defined in \cite{MP21}, the right version of the tracial approximate representability has not been known yet for this class of group actions.
 In the setting of finite groups,
Phillips in \cite{Ph11} defined the tracial analog of approximate representability under
the name ``tracial approximate representability” and proved that if $\alpha \colon G \to \Aut (A)$ is an action of 
a finite abelian group $G$ on an infinite-dimensional simple separable unital C*-algebra $A$ such that 
$C^*(G, A, \alpha)$ is simple, 
then $\alpha$
has the tracial Rokhlin property if and only if its dual $\widehat{\alpha}$ is tracially approximately
representable, and vice versa.
Nevertheless, the existence of projections may create an obstruction. Namely,
so many interesting C*-algebras are projectionless and, therefore, group actions on this class of C*-algebras
 with the Rokhlin property or even the tracial Rokhlin property are rare. 
 The need for filling this void has led to have a weaker version 
 of the tracial Rokhlin property under the name ``weak tracial Rokhlin property''. 
The idea is to replace Rokhlin projections by positive contractions. This idea
was introduced by Archey in \cite{ArThesis} and Hirshberg and Orovitz in \cite{HO13} with different names.
  Further,
 the interconnections among strong outerness, the weak tracial Rokhlin property, and finite Rokhlin have been thoroughly studied in  \cite{GHV22}. In particular, it has been established that for classifiable C*-algebras with compact trace space, the weak tracial Rokhlin property is equivalent to strong outerness.
This can be seen as a C*-algebraic counterpart of Jones' theorem concerning outer actions on the hyperfinite $II_1$-factor and greatly contributes to the significance and importance of the weak tracial Rokhlin property and its dual. 
Also, the existence of group actions with the weak tracial property (see \cite{HO13}) 
and the question of discovering an appropriate  definition of a ``weak tracial" analog of the
approximate representability for finite group actions, asked by Phillips,
 are part of our motivation 
 to think of  this class of the group actions and their duals.

In this paper,  we define a ``weak tracial” analog of the
approximate representability for finite group actions (see Definition~\ref{Week.Ap.Rep.Def}) 
and show in Theorem~\ref{MainThmCoaction} that an outer action of a finite abelian group 
on an infinite-dimensional simple unital C*-algebra
has the weak tracial Rokhlin property if and only if its dual is weakly tracially
approximately representable, and vice versa.
This is in fact a general duality result for finite group actions and generalizes the results of Izumi \cite{Iz1}
 in the setting of simple C*-algebras and the results of Phillips \cite{Ph14} on the dual actions.
Also, we give examples of finite group actions which are weakly tracially approximately representable but not tracially approximately representable 
(see
the discussion after Theorem~\ref{MainThmCoaction} 
and 
Example~\ref{Exam2})
and also some other examples with or without being weakly tracially approximately representable. (See
Example~\ref{Exam1}, 
Example~\ref{EXAmp3},  and Example~\ref{EXAmp4}.)

After this work was posted on arXiv, the definition of  the weak tracial approximate representability (Definition~\ref{Week.Ap.Rep.Def})
has been generalized  to non-finite discrete  groups and separable C*-algebras in \cite{Lee22}.   However, the main result of \cite{Lee22}, Theorem~3.21, is identical to  Theorem~\ref{MainThmCoaction} in this paper, with somewhat different methods. 
\subsection*{Acknowledgments}
Some parts of this work were carried out from May 2021 to June 2022 during the time that the  author was a postdoctoral fellow at the Czech Academy of Sciences. The author was funded by Czech Science Foundation GJ20-17488Y and RVO: 67985840. He is thankful to that institution for its hospitality, with special thanks to 
Andrey Krutov, Beata Kubiś, Wiesław Kubiś, Filip Rydlo, and Karen Strung.
Also, the author worked on some of this research while  visiting the Fields Institute from July 2023 to December 2023. He appreciates their hospitality and wants to give special thanks to  Bryan Eelhart, Miriam Schoeman, and Kirsten Vanstone.

The author would like to thank N. Christopher Phillips
for a number of  productive discussions. Lemma~\ref{L_0X12_ModPasn} and 
Lemma~\ref{L_0X19_Norm_fa_x_fa} were suggested by him, and are used here with his permission.
The author is also grateful to Bhishan Jacelon for a careful reading of a preliminary version of this paper,  and in particular finding a number of misprints.
Finally, 
the author sincerely thanks the referee for their a careful reading of the manuscript and for offering valuable comments that lead to Example \ref{Exam2}.
\section{Preliminaries }\label{Sec_Prelim}
%
In this  section, we begin by fixing some notation.
We then collect briefly 
some information on the Cuntz semigroup,
 the radius of comparison, and group actions with the weak tracial Rokhlin property  for easy reference and the convenience of the reader.
\begin{ntn}
Throughout, 
if $A$ is a \ca, or if $A = M_{\infty} (B)$
for a C*-algebra~$B$, we write $A_{+}$ for the
set of positive elements of $A$.
For $a \in A_+$, we denote by $(a -\ep)_+$ the element of $C^*(a)$ corresponding to the function $\max(0,t-\ep)$ on the spectrum of $a$.
 Also, we denote by $\cK$ the algebra of compact operators on a separable
and infinite-dimensional Hilbert space.
For an action $\alpha$ of a finite group $G$ on a unital C*-algebra $A$, 
we denote by $A^{\alpha}$ the fixed point algebra, given by
\[
A^{\alpha} = \big\{ a \in A \colon
\alpha_g (a) = a \mbox{ for all } g \in G \big\}.
\]
\end{ntn}
The following definition is originally from~\cite{Cun78}.
\begin{dfn}\label{Cuntz_def_property}
Let $A$ be a C*-algebra.
For any $a, b \in A_{+}$,
we say that $a$ is {\emph{Cuntz subequivalent to~$b$ in~$A$}},
written $a \precsim_{A} b$,
if there is a sequence $(y_n)_{n = 1}^{\infty}$ in $A$
such that
\[
\limi{n} y_n b y_n^* = a.
\]
If $a \precsim_{A} b$ and $b \precsim_{A} a$, 
we say that $a$ and $b$ are {\emph{Cuntz equivalent in~$A$}} and write $a \sim_{A} b$.
This relation is an equivalence relation. 
\end{dfn}
Part (\ref{PhiB.Lem_18_4_11}) of the following lemma is known as 
R{\o}rdam's Lemma. It is Proposition 2.4 of \cite{Ror92} or Theorem 2.7 of \cite{EP22}. 
 Part (\ref{PhiB.Lem_18_4_10.a}) is Lemma 2.5 of \cite{EP22}
 and part (\ref{Item_9420_LgSb_1_6})  follows immediately from (\ref{PhiB.Lem_18_4_10.a}) because
 $\| (a - \lambda) - (b- \lambda) \| < \varepsilon$.
\begin{lem}\label{PhiB.Lem_18_4}
Let $A$ be a \ca.
\begin{enumerate}
\item\label{PhiB.Lem_18_4_11}
Let $a, b \in A_{+}$.
Then \tfae:
\begin{enumerate}
\item\label{PhiB.Lem_18_4_11.a}
$a \precsim_A b$.
\item\label{PhiB.Lem_18_4_11.b}
$(a - \eta)_{+} \precsim_A b$ for all $\eta > 0$.
\item\label{PhiB.Lem_18_4_11.c}
For every $\eta > 0$, there is $\dt > 0$ such that
$(a - \eta)_{+} \precsim_A (b - \dt)_{+}$.
\end{enumerate}
\item\label{PhiB.Lem_18_4_10}
Let $a, b \in A_{+}$ and let $\ep > 0$.
If $\| a - b \| < \ep$, then:
\begin{enumerate}
\item\label{PhiB.Lem_18_4_10.a}
$(a - \ep)_{+} \precsim_A b$.
\item\label{Item_9420_LgSb_1_6}
For any $\ld > 0$,
we have $(a - \ld - \ep)_{+} \precsim_A (b - \ld)_{+}$.
\end{enumerate}
\end{enumerate}
\end{lem}
The following is Lemma~2.6 of \cite{AGP19}.
\begin{lem}\label{Lem.ANP.Dec.18}
Let $A$ be a unital \ca, let $a, g \in A$
satisfy $0 \leq a, g \leq 1$,
and let $\ep_1, \ep_2 \geq 0$.
Then
\[
\big( a - (\ep_1 + \ep_2) \big)_{+}
\precsim_{A} \big( (1 - g ) a (1 - g ) - \ep_1 \big)_{+}
         \oplus \Big( g - \frac{\ep_2 }{2} \Big)_{+}.
\]
\end{lem}
\begin{lem}
\label{lem210912}
Let $A$ be a unital C*-algebra, let $a\in A_+$ with $\|a \| \leq 1$, and let $\ep \in (0, 1)$.
Let $g_{\ep} \colon [0, 1] \to [0, 1]$ be the continuous function by
\[
g_{\ep} (\ld)
 = \begin{cases}
   \frac{1}{1-\ep} t & \hspace*{1em} t \leq 1 - \ep
        \\
   1  & \hspace*{1em} t \geq 1-\ep.
\end{cases}
\]
Then $(1 - a -\ep)_+ \sim_A 1 - g_{\ep} (a)$.
\end{lem}
\begin{proof}
It can be easily seen by a functional calculus argument.
\end{proof}
%
\subsection{Approximation Lemmas}
This subsection contains several approximation lemmas
which will be needed in the following sections.
\begin{lem}
\label{ArchyLemmaAp20}
Let $M\in (0, \infty)$ and  let $f \colon [0,\ M] \to \mathbb{C}$ be a  continuous function.  Then for every $\ep> 0$ there
exists $\dt> 0$ such that whenever $A$ is a C*-algebra and $x, z \in A$ satisfy
$
0 \leq x \leq M$,
$\| z \| \leq M$,
and
$\| xz - zx \| < \dt$,
then $\|f(x)z - z f(x)\| < \ep$.
\end{lem}
\begin{proof}
The proof is as same as the proof of Lemma~2.5 of \cite{AP20}, except we allow here to use $M$. This change makes no difference. 
\end{proof}
\begin{lem}\label{approxlemmaArch}
Let $M\in (0, \infty)$,  let $f \colon [0,\ M] \to \mathbb{C}$ be a  continuous function,
and let $M>0$. Then for every $\ep> 0$ 
 there exists $\dt > 0$ such
that whenever $A$  is a  C*-algebra and  $x, y\in A_+$ satisfying $\|x\|, \|y \| \leq M$ and
$\|x - y\| < \dt$, then $\| f(x) - f(y)\|< \ep$.
\end{lem}
\begin{proof}
The case $M = 1$ is Lemma~VI.11 in \cite{ArThesis}. The proof of this version is the same.
\end{proof}
The following lemma is preparation for Lemma~\ref{L_0X19_Norm_fa_x_fa}.
\begin{lem}\label{L_0X19_BigNorm}
Let $H$ be a Hilbert space,
let $a \in L (H)$ satisfy $0 \leq a \leq 1$,
and let $\xi \in H$ satisfy $\| \xi \| = 1$.
Then $\| a \xi - \xi \| \leq \sqrt{1 - \| a \xi \|^2}$.
\end{lem}

\begin{proof}
We have
$\| a^{1/2} \xi \| \geq \| a^{1/2} a^{1/2} \xi \| = \| a \xi\|$.
So $\langle a \xi, \xi \rangle \geq \| a \xi\|^2$.
Now we have
\[
\| a \xi - \xi \|^2
  = \| a \xi \|^2 - 2 \langle a \xi, \xi \rangle + \| \xi \|^2
  \leq \| a \xi \|^2 - 2 \| a \xi\|^2 + 1
  = 1 - \| a \xi \|^2.
\]
The result follows.
\end{proof}

\begin{lem}\label{L_0X19_Norm_fa_x_fa}
Let $f \colon [0, 1] \to \C$ be a \cfn{} such that
$f (1) = 1$.
Then for every $\ep > 0$ there is $\dt > 0$ such that whenever
$A$ is a \ca, $a, x \in A_{+}$ satisfy
$\| a \| \leq 1$, $\| x \| \leq 1$, and
$\| a x a \| > 1 - \dt$,
then $\| f (a) x f (a) \| > 1 - \ep$.
\end{lem}

\begin{proof}
Let $\ep > 0$.
\Wolog, $\ep < 1$.
Set $M = \sup_{\ld \in [0, 1]} |f (\ld)|$.
Then $M \geq 1$.
Choose a polynomial function $g$ such that
$g (1) = 1$ and for $\ld \in [0, 1]$ we have
\[
| g (\ld) - f (\ld) | < \frac{\ep}{3 (2 M + 1)}.
\]
Write
$g (\ld) = \sum_{k = 0}^n \af_k \ld^k$
with $\af_0, \af_1, \ldots, \af_n \in \C$.
Then
\[
\sum_{k = 0}^n \af_k = 1
\andeqn
\sup_{\ld \in [0, 1]} |g (\ld)| \leq M + 1.
\]
Set $R = \sum_{k = 0}^n k |\af_k|$.
Choose $\dt > 0$ so small that
\[
2 (M + 2) R \dt^{1/2} + 3 (M + 1) \dt^{1/4}
  < \frac{\ep}{3}.
\]
In particular, $\dt < 1$.

Let $A$, $a$, and $x$ be as in the hypotheses,
with this choice of~$\dt$.
We can assume $A \S L (H)$ for some Hilbert space~$H$,
and if $A$ is unital, we can assume its identity is the
identity operator on~$H$.

Since $\| a x a \| > 1 - \dt$,
there is $\xi \in H$ such that
$\| \xi \| = 1$ and $\| a x a \xi \| > 1 - \dt$.
Clearly $\| a \xi \| > 1 - \dt$,
so Lemma~\ref{L_0X19_BigNorm} implies
\[
\| a \xi - \xi \|
 < \sqrt{1 - (1 - \dt)^2}
 = \sqrt{2 \dt - \dt^2}
 < 2 \dt^{1/2}.
\]
Similarly $\| a x a \xi \| \leq \| x a \xi \|$, so
\[
\| x \xi \|
 \geq \| x a \xi \| - \| x \| \| a \xi - \xi \|
 > 1 - \dt - 2 \dt^{1/2}
 > 1 - 3 \dt^{1/2},
\]
and Lemma~\ref{L_0X19_BigNorm} implies
\[
\| x \xi - \xi \|
 < \sqrt{1 - \bigl( 1 - 3 \dt^{1/2} \bigr)^2}
 = \sqrt{6 \dt^{1/2}- 9 \dt}
 < 3 \dt^{1/4}.
\]

For $k \in \Nz$,
we have
\[
\| \xi - a^k \xi \|
 \leq \| \xi - a \xi \| + \| a \| \cdot \| \xi - a \xi \|
          + \cdots + \| a^{k - 1} \|  \cdot \| \xi - a \xi \|
 \leq 2 k \dt^{1/2}.
\]
Therefore, using $\sum_{k = 0}^n \af_k = 1$,
\[
\| \xi - g (a) \xi \|
 \leq \sum_{k = 0}^n | \af_k | \| \xi - a^k \xi \|
 \leq \sum_{k = 0}^n 2 k \dt^{1/2} |\af_k|
 = 2 R \dt^{1/2}.
\]
Now, using this, we get
\[
\begin{split}
\| g (a) x g (a) \xi - \xi \|
& \leq \| g (a) \| \cdot \| x \| \cdot \| g(a) \xi - \xi \|
        + \| g (a) \| \cdot \| x \xi - \xi \|
        + \| g(a) \xi - \xi \|
\\
& < (M + 1) \cdot 2 R \dt^{1/2}
    + (M + 1) \cdot 3 \dt^{1/4} + 2 R \dt^{1/2}
  < \frac{\ep}{3}.
\end{split}
\]
It follows that $\| g (a) x g (a) \xi \| > 1 - \frac{\ep}{3}$,
so $\| g (a) x g (a) \| > 1 - \frac{\ep}{3}$.

The choice of $g$ implies that
\[
\| g (a) - f (a) \| \leq \frac{\ep}{3 (2 M + 1)},
\]
so we now have
\[
\begin{split}
\| f (a) x f (a) \|
& > \| g (a) x g (a) \| - \| g (a) - f (a) \| \cdot \| x \| \cdot \| g (a) \|
\\       
& \qquad - \| f (a) \| \cdot \| x \|  \cdot \| g (a) - f (a) \|
\\
& > 1 - \frac{\ep}{3} - \frac{\ep (M + 1)}{3 (2 M + 1)}
        - \frac{\ep M}{3 (2 M + 1)}
  \geq 1 - \ep.
\end{split}
\] 
This completes the proof.
\end{proof}
\subsection{Weak Tracial Rokhlin Property}
In this subsection, we recall the definition of the weak tracial Rokhlin property and prove some equivalent conditions for it.
\begin{dfn}\label{W_T_R_P_def}
Let $G$ be a finite group,
let $A$ be a simple unital \ca,
and let $\alpha \colon G \to \Aut (A)$  be an action of
$G$ on $A$.
We say that $\alpha$ has the
\emph{weak tracial Rokhlin property} if for every $\ep > 0$,
every finite set $F \subset A$, and every positive
element $x \in A$ with $\| x \| = 1$,
there exist orthogonal positive contractions
$f_g \in A$ for $g \in G$ such that,
with $f = \sum_{g \in G} f_g$, the following hold:
\begin{enumerate}
\item\label{Def.w.t.r.p.a}
$\| a f_g - f_g a \| < \ep$ for all $g \in G$ and all $a \in F$.
\item\label{Def.w.t.r.p.b}
$\| \alpha_{g} ( f_h ) - f_{g h} \| < \ep$ for all $g, h \in G$.
\item\label{Def.w.t.r.p.c}
$1 - f \precsim_A x$.
\item\label{Def.w.t.r.p.d}
$\|  f x f \| > 1 - \ep$.
\end{enumerate}
\end{dfn}
In the following lemma,  we show that orthogonality of the contractions $f_g$ for $g \in G$ in Definition~\ref{W_T_R_P_def} can be replaced by approximate orthogonality. Further, we show that  $f$ in Definition~\ref{W_T_R_P_def} can be chosen in the fixed point algebra $A^\alpha$,  at the cost of replacing orthogonality by
 approximate orthogonality of the contractions $f_g$ for $g \in G$.  This will be useful in the proofs of Proposition~\ref{ProWTRP->WTAR} and Proposition~\ref{ProWTAR->WTRP}.
\begin{lem}\label{invariant_contractions}
Let $G$ be a finite group,
let $A$ be an infinite-dimensional simple unital \ca,
and let $\alpha \colon G \to \Aut (A)$  be an action of
$G$ on $A$.
Then the following are equivalent:
\begin{enumerate}
\item\label{MainS1}
 $\alpha$ has the weak tracial Rokhlin property.
 \item\label{MainS2}
For every $\ep > 0$,
every finite set $F \subset A$, and every positive
element $x \in A$ with $\| x \| = 1$,
there exist positive contractions
$f_g \in A$ for $g \in G$ such that,
with $f = \sum_{g \in G} f_g$,
the following hold:
\begin{enumerate}
\item\label{W.T.R.P.66_0}
$\| f_g f_h \| < \ep$
for all $g, h \in G$.
\item\label{W.T.R.P_11_0}
$\| a f_g - f_g a \| < \ep$
for all $g \in G$ and all $a \in F$.
\item\label{W.T.R.P_22_0}
 $\| \alpha_{g} ( f_h ) - f_{g h} \| < \ep$ for all $g, h \in G$.
\item\label{W.T.R.P.33_0}
$(1 - f - \ep)_{+} \precsim_A x$.
\item\label{W.T.R.P.44_0}
$\| f x f \| > 1 - \ep$.
\end{enumerate}
\item\label{MainS3}
For every $\ep > 0$,
every finite set $F \subset A$, and every positive
element $x \in A$ with $\| x \| = 1$,
there exist positive contractions
$f_g \in A$ for $g \in G$ such that,
with $f = \sum_{g \in G} f_g$,
the following hold:
\begin{enumerate}
\item\label{W.T.R.P.66}
$\| f_g f_h \| < \ep$
for all $g, h \in G$.
\item\label{W.T.R.P_11}
$\| a f_g - f_g a \| < \ep$
for all $g \in G$ and all $a \in F$.
\item\label{W.T.R.P_22}
 $\alpha_{g} ( f_h ) = f_{g h}$ for all $g, h \in G$.
\item\label{W.T.R.P.33}
$(1 - f - \ep)_{+} \precsim_A x$.
\item\label{W.T.R.P.44}
$\| f x f \| > 1 - \ep$.
\item\label{W.T.R.P.55}
 $f \in A^{\alpha}$  and $\| f \| = 1$.
\end{enumerate}
\setcounter{TmpEnumi}{\value{enumi}}
\end{enumerate}
\end{lem}
\begin{proof}
It follows from Lemma~3.3 of \cite{AGP19} that (\ref{MainS1}) implies (\ref{MainS3}). Also,   (\ref{MainS3}) clearly implies (\ref{MainS2}). So it suffices to prove that (\ref{MainS2}) implies (\ref{MainS1}).
To prove it,
 let $\ep\in (0, 1)$, let $F$ be a finite set in $A$, let $x \in A_+$ with $\| x \|=1$, and let $g_\ep$ be as in Lemma~\ref{lem210912}.
Assume $\card (G)= n$ for some $n \in \N$ and set $M= \max \left (1, \max\{ \| a \| \colon a \in F\} \right)$. 
Using Lemma~2.5.12 of \cite{LinBook}, Lemma~\ref{ArchyLemmaAp20}, Lemma~\ref{approxlemmaArch},
and Lemma~\ref{L_0X19_Norm_fa_x_fa},
 we can choose $\dt_1, \dt_2>0$ such that the following hold:
\begin{enumerate}
\setcounter{enumi}{\value{TmpEnumi}}
\item\label{approx1}
$\dt_1 < \dt_2 < \ep$.
\item\label{approx2}
If  $a_1, \ldots, a_n \in A_+$ 
with $\| a_j \|\leq 1$ for $j=1, \ldots, n$ such that 
$\| a_j a_k\| < \dt_1$ when $j \neq k$, then there are 
$b_1, \ldots, b_n \in A_+$ such that 
$b_j b_k = 0$ when $j \neq k$, $\| b_j \| \leq 1$, and 
$\| a_j - b_j \| < \frac{\dt_2}{12n^2M}$
for $j=1, \ldots, n$.
\item
\label{approx3}
If $x, z \in A$ satisfy
$
0 \leq x \leq 1,
\| z \| \leq M,
$
and
$\| xz - zx \| < \dt_2$,
then 
\[
\|g_{\ep} (x)z - z g_{\ep} (x)\| < \ep.
\]
\item\label{approx4}
If $x, z\in A_+$ with $\|x\|, \|z \| \leq 1$ and
$\|x - z\| < \dt_2$, then 
\[
\| g_{\ep}(x) - g_{\ep}(z)\|< \ep.
\]
\item
\label{approx5}
If $x, z \in A_{+}$ satisfy
$\| x\| \leq 1$, $\| z \| \leq 1$, and
$\| x z x\| > 1 - \dt_2$,
then 
\[
\| g_{\ep} (x) z g_{\ep} (x) \| > 1 - \ep.
\]
\setcounter{TmpEnumi}{\value{enumi}}
\end{enumerate}

Set 
$
\ep' = \frac{\dt_1}{6M^2 n}.
$
Using the conditions in Part (\ref{MainS2}) of the lemma with $\ep'$ in place of $\ep$,
 $F$, and $x\in A_+$ as given, 
we can find  positive contractions
$f_g \in G$ in $A$ for $g \in G$ such that, 
with $f = \sum_{g \in G} f_g$,  the following hold:
\begin{enumerate}
\setcounter{enumi}{\value{TmpEnumi}}
\item\label{W.T.R.P.66.a_0}
$\| f_g f_h \| < \ep'$
for all $g, h \in G$.
\item\label{W.T.R.P_11.b_0}
$\| a f_g - f_g a \| < \ep'$
for all $g \in G$ and all $a \in F$.
\item\label{W.T.R.P_22.c_0}
 $\| \alpha_{g} ( f_h ) - f_{g h}\| < \ep'$ for all $g, h \in G$.
\item\label{W.T.R.P.33.d_0}
$(1 - f - \ep')_{+} \precsim_A x$.
\item\label{W.T.R.P.44.e_0}
$\| f x f \| > 1 - \ep'$.
\setcounter{TmpEnumi}{\value{enumi}}
\end{enumerate}
By (\ref{approx2}), there are contractions $b_g \in A_+$ for $g \in G$ such that:
\begin{enumerate}
\setcounter{enumi}{\value{TmpEnumi}}
\item\label{Orthogonalelements1}
$b_g b_h = 0$ for all $g, h \in G$ with $g \neq h$.
\item\label{Orthogonalelements2}
$\| f_g - b_g \| < \frac{\dt_2}{12 n^2 M}$ for $g \in G$.
\setcounter{TmpEnumi}{\value{enumi}}
\end{enumerate}
Now, for every $g \in G$, set
\[
b= \sum_{g \in G} b_g,
\qquad
d_g = g_\ep (b_g),
\qquad
\mbox{and}
\qquad
d = g_\ep (b).
\]
By Lemma~VI.10 of \cite{ArThesis}, it is clear that $d=\sum_{g \in G} d_g$. Now, we claim that the following hold:
\begin{enumerate}
\setcounter{enumi}{\value{TmpEnumi}}
\item\label{Def'.w.t.r.p.a_0}
$d_g$ for $g \in G$ are mutually orthogonal positive contractions.
\item\label{Def'.w.t.r.p.a}
$\| a d_g - d_g a \| < \ep$ for all $g \in G$ and all $a \in F$.
\item\label{Def'.w.t.r.p.b}
$\| \alpha_{g} ( d_h ) - d_{g h} \| < \ep$ for all $g, h \in G$.
\item\label{Def'.w.t.r.p.c}
$1 - d \precsim_A x$.
\item\label{Def'.w.t.r.p.d}
$\|  d x d \| > 1 - \ep$.
\setcounter{TmpEnumi}{\value{enumi}}
\end{enumerate}
Since $b_g$ for $g \in G$ are mutually orthogonal positive contractions, Part (\ref{Def'.w.t.r.p.a_0}) is immediate.

To prove (\ref{Def'.w.t.r.p.a}), we use (\ref{Orthogonalelements2}) and (\ref{W.T.R.P_11.b_0}) at the second step to get
\[
\| b_g a - a b_g\| \leq \|b_g -f_g \| \cdot \| a\| + \| f_g a - a f_g\| + \| a\| \cdot \| f_g - b_g\|
<
\frac{\dt_2}{3}
+
\ep'
+
\frac{\dt_2}{3}
< \dt_2.
\]
Then, by (\ref{approx3}), we have
$\|d_g a - a d_g \| <\ep$.

To prove (\ref{Def'.w.t.r.p.b}), we use (\ref{Orthogonalelements2}) and (\ref{W.T.R.P_22.c_0}) at the second step to get
\[
\| \alpha_g (b_h)- b_{gh}\| 
\leq
\|\alpha_g (b_h) - \alpha_g (f_h) \|
+
\| \alpha_g (f_h) - f_{gh} \|
+
\| f_{gh} - b_{gh}\|
<
\frac{\dt_2}{3}
+
\ep'
+
\frac{\dt_2}{3}
< \dt_2.
\]
Then, by (\ref{approx4}), we have
$\| \alpha_g(d_h)- d_{gh}\| < \ep$.

To prove (\ref{Def'.w.t.r.p.c}), we use (\ref{Orthogonalelements2}) at the 
third step to get
\begin{equation}\label{Eq20210914}
\|(1 - f) - (1 - b) \| =
\| f  - b \|
\leq
\sum_{g\in G} \|f_g - b_g\|
<
\frac{\dt_2}{4n} <\frac{\dt_2}{2}.
\end{equation}
Using  Lemma~\ref{lem210912} with $b$ in place of $a$ at the second step, 
using Lemma~\ref{PhiB.Lem_18_4_10}(\ref{Item_9420_LgSb_1_6}) and (\ref{Eq20210914}) at the fourth step, 
and using (\ref{W.T.R.P.33.d_0})
at the last step, we get
\begin{align*}
1- d =1 - g_{\ep} (b) \sim_A (1 - b -\ep)_+ 
&\precsim_A 
(1 - b -\dt_2)_+ 
  \\ & \precsim_A 
  \left(1 - f -\tfrac{\dt_2}{2}\right)_+ 
\\&\precsim_A 
 \left(1 - f -\ep'\right)_+ \precsim_A x.
\end{align*}

To prove (\ref{Def'.w.t.r.p.d}), we use (\ref{Eq20210914}) and the fact that $\|f \|\leq n$ at the second step to get
\[
\|fxf - bxb \| \leq \|f - b \|\cdot \|x f \| +  \|bx\| \cdot \|f - b \| < \frac{\dt_2}{4} + \frac{\dt_2}{4} =\frac{\dt_2}{2}.
\]
Now, using this at the first step and using (\ref{W.T.R.P.44.e_0}) at the second step, we get
\[
\|bxb \|> \| fxf\| - \frac{\dt_2}{2}> 1- \ep' - \frac{\dt_2}{2}>  1 - \dt_2.
\]
Then, by (\ref{approx5}),  we have
$\|d x d \| > 1- \ep$.
This completes the proof.
\end{proof}
\section{The dual of  weakly tracially approximately representable actions}
In this section, we first give the definition of weakly tracially approximately representable actions. 
We then give a general duality theorem for a finite abelian group action when the action has
the weak tracial Rokhlin property or is weakly tracially approximately representable. 
 Further, we exhibit some examples to show that our theorems are not empty.

The following is a ``weak tracial" analog of approximate representability for finite group 
actions on infinite-dimensional simple unital C*-algebras. 
\begin{dfn}\label{Week.Ap.Rep.Def}
 Let $A$ be an infinite-dimensional simple unital \ca. An action $\alpha \colon  G \to \Aut(A)$ of a
finite group $G$ on $A$ is \emph{weakly tracially approximately representable} if for every finite
set $F \subset A$, every $\ep > 0$, and every
positive element $x \in A$ with $\|x\| = 1$, there are
 $c  \in (A^{\alpha})_+$ with $\| c \| \leq 1$ and 
 contractive elements  $s_g \in  \overline{c A c}$ 
 for $g \in G$ such that:
\begin{enumerate}
\item
\label{Week.Ap.Rep.Def.1}
$\| s_1 - c\|< \ep$ and $\|s^*_g - s_{g^{-1}}  \|< \ep$ for all $g \in G$.
\item
\label{Week.Ap.Rep.Def.2}
$\| s_g s_h - c s_{gh}\| < \ep$ for all $g, h \in G$.
\item
\label{Week.Ap.Rep.Def.3}
$\| c a - a c \| < \ep$ for all $a \in F \cup \{s_g \colon g \in G\}$.
\item
\label{Week.Ap.Rep.Def.4}
$\| \alpha_g (c a c) - s_g a  s^*_g \|< \ep$ for all $a \in F$ and all $g \in G$.
\item
\label{Week.Ap.Rep.Def.5}
$\| \alpha_g (s_h) - s_{g h g^{-1}}\| < \ep$ for all $g, h \in G$.
\item
\label{Week.Ap.Rep.Def.6}
$(1 - c - \ep)_+ \precsim_A x$.
\item
\label{Week.Ap.Rep.Def.7}
$\| c x c\| > 1-\ep$.
\end{enumerate}
\end{dfn}
It is clear that Condition (\ref{Week.Ap.Rep.Def.5}) reduces to $\| \alpha_g (s_h) - s_{h}\| < \ep$ if $G$ is abelian group.
We refer to Definition~4.1 of \cite{AV20} for a tracial version of Definition~\ref{Week.Ap.Rep.Def} 
in the setting of nonabelian finite groups  and 
 Definition~3.2 of~\cite{Ph11} in the setting of abelian finite groups and separable C*-algebras.

In the following lemma, we show that in the category of  pointwise outer actions of countable discrete groups on simple C*-algebras, for every positive element,  say $a$, in the crossed product of norm one,  
there is a positive contraction in the crossed product, say $z$, 
such that $zaz^* \in A_+ \setminus \{0\}$ and its norm is bounded below.
This lemma may be known and can be considered as a standard application of Kishimoto's lemma. We include a proof as we need to know some information on the norm of $z$ and $zaz^*$ when it is used
 in the proofs of Proposition~\ref{ProWTRP->WTAR} and Proposition~\ref{ProWTAR->WTRP}. 
\begin{lem}\label{L_0X12_ModPasn}
Let $\af \colon G \to \Aut (A)$
be an action
of a  countable discrete group $G$ on a unital simple  \ca~$A$
which is pointwise outer.
Identify $A$ with a subalgebra of $C^*_{\mathrm{r}} (G, A, \af)$
in the usual way.
Let $a \in C^*_{\mathrm{r}} (G, A, \af)_{+}$
satisfy $\| a \| = 1$.
Then for every $\ep > 0$ there is $z \in C^*_{\mathrm{r}} (G, A, \af)$
such that:
\begin{enumerate}
\item\label{Item_L_0X12_ModPasn_InA}
$z a z^* \in A_{+}$.
\item\label{Item_L_0X12_ModPasn_Normz}
$\| z \| \leq 1$.
\item\label{Item_L_0X12_ModPasn_NzazSt}
$\| z a z^* \| > 1 - \ep$.
\end{enumerate}
\end{lem}

\begin{proof}
Let $\ep \in (0, 1)$. Choose $\lambda$ such that $1- \ep < \lambda < 1$. 
Set $b= (a - \lambda)_+$. 
We use Lemma~4.2 of \cite{GGNV22} to get $c \in A_+\setminus\{0\}$ such that 
$c \precsim_{ C^*_{\mathrm{r}} (G, A, \af) }  b$. Note that separability of the C*-algebra in Lemma~4.2 of \cite{GGNV22} is not needed. We may assume $\| c \| =1$. Now, we choose 
$\delta \in (0, 1)$ such that $(1- \delta) \lambda> 1- \ep$.
We use Lemma~2.7 of \cite{AGP19} to get $w_0 \in C^*_{\mathrm{r}} (G, A, \af)$ such taht 
\begin{equation}
\label{202309.05.1}
\| w_0 a w_0^* -c\|<\delta 
\qquad
\mbox{and}
\qquad
\|w_0 \| \leq \lambda^{-1/2}.
\end{equation}
Using Lemma~2.2 of \cite{KR00}, we get a contraction $d \in C^*_{\mathrm{r}} (G, A, \af)$ such that 
\begin{equation}
\label{2023.09.06.II}
dw_0 a w_0^* d^* = (c- \delta)_+.
\end{equation}
Now, set $z= \frac{1}{\| w_0\|} \cdot d w_0$. It is clear that $\|z\|\leq 1$ and 
$z a z^* \in A_+$. Now, using (\ref{2023.09.06.II}) at the first step and using the second part of (\ref{202309.05.1}) and the fact that 
$\| (c - \delta)_+\| \geq 1- \delta$ at the second step, we get
\begin{align*}
\|z b z^*\| = \frac{1}{\| w_0\|^2} \cdot \| (c- \delta)_+\| > \lambda (1- \delta) > 1- \ep.
\end{align*} 
This completes the proof. 
\end{proof}
If  a finite group action on a simple unital C*-algebra has the weak tracial Rokhlin property, 
it follows from Proposition 3.2 of \cite{FG17} that $\alpha$ is pointwise outer and, therefore, Lemma~\ref{L_0X12_ModPasn} can be applied.
\begin{ntn}\label{N_9408_StdNotation_CP}
Let $A$ be a unital C*-algebra and let $\af \colon G \to \Aut (A)$ be an action
of a finite group $G$ on $A$. 
For $g \in G$,
we let $u_g$ be the element of $C_{\mathrm{c}} (G, A, \af)$
which takes the value $1$ at $g$
and $0$ at the other elements of~$G$.
We use the same notation for its image in $C^* (G, A, \af)$.
Also, for each $g \in G$, we define the map 
$ E_g \colon  \CGAa \to A$ by 
$E_g (a) = a_g$,
where $a= \sum_{g \in G} a_g u_g$.
\end{ntn}
\begin{rmk}\label{rmk211005}
Let $G$ be an abelian finite group. Then:
\begin{enumerate}
\item
For every $\chi \in \widehat{G}$, we have
$\sum_{g\in G} \chi (g) = 
\begin{cases}
\card (G) & \chi=1\\
0 &         \chi\neq 1.
\end{cases}
$
\item
$\sum_{\chi \in \widehat{G}} \chi (g) = 
\begin{cases}
\card (G) & g=1\\
0 &         g \neq 1.
\end{cases}
$
\end{enumerate}
\end{rmk}
\begin{prp}\label{ProWTRP->WTAR}
Let $A$ be an infinite-dimensional simple unital C*-algebra and let $G$ be a finite abelian group.
Let $\alpha \colon G \to \Aut (A)$ be an action of  $G$ on $A$ which has the weak tracial Rokhlin property.
Then $\widehat{\alpha}$ is weakly tracially approximately representable.
\end{prp}
\begin{proof}
Let $\ep\in (0, 1)$, let $x \in C^*(G, A, \alpha)_+$ with $\| x \|=1$, and let $F \subset C^*(G, A, \alpha)$ be a finite set.
 We use Lemma~\ref{L_0X12_ModPasn} with $x$ in place of $a$ to choose  
 $z \in C^*(G, A, \alpha)$ such that 
 \begin{equation}
\label{Eq.b.2021.10.11}
zxz^* \in A_+,
\qquad
\| z\| \leq 1,
\qquad 
\mbox{ and }
\qquad
\|zxz^* \| >1 - \frac{\ep}{8}.
\end{equation}

 Without loss of generality,  
 $F= F_0 \cup \{u_g \colon g \in G\}$ for some finite subset $F_0 \subset A$ such that 
 $\| a \| \leq 1$ for all $a \in F_0$.
 Set 
\[
F_1 = F_0 \cup \left\{E_g (z), E_g (z^*)  \colon g \in G\right\},
\hspace*{.1em}
y = \frac{1}{\| zxz^*\|} \cdot zxz^*,
\hspace*{.1em}
\mbox{ and }
\hspace*{.1em}
\ep'=\frac{\ep}{48 (1+ \tfrac{\ep}{8})\card(G)^2}.
\] 
Applying Lemma~\ref{invariant_contractions} with $F_1$ and $y$ as given and  $\ep'$ in place of $\ep$,
we get positive contractions
$f_g \in A$ for $g \in G$ such that,
with $f = \sum_{g \in G} f_g$, the following hold:
\begin{enumerate}
\item\label{W.T.R.P.66.a}
 $\| f_g f_h \| < \ep'$
for all $g, h \in G$.
\item\label{W.T.R.P_11.b}
 $\| a f_g - f_g a \| < \ep'$
for all $g \in G$ and all $a \in F_1$.
\item\label{W.T.R.P_22.c}
 $ \alpha_{g} ( f_h ) = f_{g h}$ for all $g, h \in G$.
\item\label{W.T.R.P.33.d}
$(1 - f - \ep')_{+} \precsim_A y$.
\item\label{W.T.R.P.44.e}
$\| f y f \| > 1 - \ep'$.
\item\label{W.T.R.P.55.f}
 $f \in A^{\alpha}$ and $\| f \| =1$.
\setcounter{TmpEnumi}{\value{enumi}}
\end{enumerate}
Now for every $\tau \in \widehat{G}$, define
\[
s_{\tau} = \sum_{g \in G} \overline{\tau(g)} f_g,
\qquad
t_{\tau}= \frac{1}{1+\tfrac{\ep}{8}} \cdot s_{\tau},
\qquad
\mbox{and}
\qquad
c= \frac{1}{1+\tfrac{\ep}{8}} \cdot f.
\]
We claim that:
\begin{enumerate}
\setcounter{enumi}{\value{TmpEnumi}}
\item\label{A.00}
$ t_{\tau} \in \overline{cAc}$ for all $\tau \in \widehat{G}$.
\item\label{A.0}
$\| t_{\tau}\| \leq 1$ for all $\tau \in \widehat{G}$ and $c \in A^{\alpha} \setminus \{0\}$ with $\| c \|\leq 1$.
\item
\label{A.1}
$ t_1 = c$ and $t^*_{\tau}= t_{{\tau}^{-1}}$ for all $\tau \in \widehat{G}$.
\item
\label{A.3}
$\| t_{\tau} t_{\sigma} - c t_{\tau \sigma} \| < \ep$ for all $\tau, \sigma \in \widehat{G}$.
\item
\label{A.4}
$\| c t_{\tau}  -  t_{\tau} c \| < \ep$ for all $\tau \in \widehat{G}$.
\item
\label{A.5}
$\| c v  -  v c \| < \ep$ for all $v \in F$.
\item
\label{A.6}
$\widehat{\alpha}_{\tau} (t_{\sigma}) = t_{\sigma}$ for all $\tau, \sigma \in \widehat{G}$.
\item
\label{A.7}
$\| \widehat{\alpha}_{\tau} (cvc) -  t_{\tau} v t^*_{\tau} \| < \ep$ for all $\tau \in \widehat{G}$ and $v \in F$.
\item
\label{A.8}
$(1-c- \ep)_+ \precsim_{C^*(G, A, \alpha)} x$.
\item
\label{A.9}
$\| c x c\|> 1- \ep$.
\setcounter{TmpEnumi}{\value{enumi}}
\end{enumerate}
Since $f_g \in \overline{f A f}$, 
 it follows that   
 $s_{\tau} \in \overline{f A f}$ for all $\tau \in \widehat{G}$ and, therefore, (\ref{A.00}) is immediate.
 
 To prove (\ref{A.0}), it suffices to estimate $\| s_{\tau}\|$ for $\tau \in \widehat{G}$.  
Set $S = \{ (g, g) \colon g \in G \} \subseteq G^2$ and use (\ref{W.T.R.P.66.a}) at the third step to get
\begin{equation}
\label{E1.20201002}
\left\|
 f^2
 - 
\sum_{g \in G} f^2_g
\right\|
=
\left\|
\sum_{g, h \in G} f_g f_h -  \sum_{g \in G} f^2_g
\right\|
\leq
\left\|
\sum_{g, h \notin S} f_g f_h 
\right\|
\leq 
\card(G)^2\ep'.
\end{equation}
Now, using (\ref{E1.20201002}), (\ref{W.T.R.P.66.a}), and 
the fact that $|\tau(g)|=1$ for all $g\in G$ at the second step, we estimate
\begin{align*}
\| s_{\tau} s^{*}_{\tau} - f^2 \|
&\leq 
\left\| 
\sum_{g, h \notin S} \overline{\tau (g)} \tau (h) f_g f_h
\right\|
+
\left\| 
\sum_{g \in G} |\tau (g)|^{2}  f^2_g
-
f^2
\right\|
\\
&<
\card(G)^2\ep' + \card(G)^2 \ep'<\frac{\ep}{8}.
\end{align*} 
Therefore, using this at the second step and the second part of (\ref{W.T.R.P.55.f}) at the third step,
\[
\| s_{\tau} \|^2= \| s_{\tau} s^*_{\tau} \| < \| f \|^2 + \frac{\ep}{8}= 1 + \frac{\ep}{8}.
\]
This relation implies that
$\| s_{\tau} \| <  1+ \frac{\ep}{8}$.

Part (\ref{A.1}) is easy to check.

To prove (\ref{A.3}), 
we use (\ref{W.T.R.P.66.a}) at the second step to estimate
\begin{align*}
\| s_{\tau} s_{\sigma} - f s_{\tau \sigma}\|
&\leq 
\left\| 
\sum_{g, h \notin S} \overline{\tau (g)} \overline{\sigma (h)} f_g f_h
\right\|
+
\left\| 
\sum_{g \in G} \overline{\tau \sigma (g)}  f^2_g
-
\sum_{g, h \in G} \overline{\tau \sigma (g)}  f_h f_g
\right\|
\\
&<
\card(G)^2\ep' + \card(G)^2 \ep'= 2\card(G)^2\ep'.
\end{align*}
This relation implies that
$\| t_{\tau} t_{\sigma} - c t_{\tau \sigma}\|
<\ep$.

To prove (\ref{A.4}), we use (\ref{W.T.R.P.66.a}) at the second step to get
\begin{align*}
\| f s_{\tau}  -  s_{\tau} f\|
&\leq 
\left\| 
\sum_{g, h \notin S} \overline{\tau (g)}  f_h f_g
\right\|
+
\left\| 
\sum_{g \in G} \overline{\tau (g)} f^2_g
-
\sum_{g, h \in G} \overline{\tau(h)}    f_h f_g
\right\|
\\
&<
\card(G)^2\ep' + \card(G)^2 \ep'= 2\card(G)^2 \ep'.
\end{align*}
This relation implies that
$\| c t_{\tau}  -  t_{\tau} c\|<\ep$.

We prove (\ref{A.5}). Since $c \in A^{\alpha}$,  it follows that
$u_g c = c u_g$  for all $g \in G$.
For $a \in F_1$, we use (\ref{W.T.R.P_11.b}) at the second step to get
\begin{equation}
\label{Eq5.20201011}
\| a f - f a \| 
\leq 
\sum_{g \in G} \| af_g - f_g a  \|
< \card(G) \ep'.
\end{equation}
This relation implies that
$\| a c - c a \| 
< \frac{\card(G)}{1+\tfrac{\ep}{8}} \ep'<\ep$ and, therefore, 
 (\ref{A.5}) now follows.

To prove (\ref{A.6}), for all $\tau, \sigma \in \widehat{G}$, we have
\begin{align*}
\widehat{\alpha}_{\tau} (t_{\sigma})
=
\frac{1}{1+ \tfrac{\ep}{8}}\cdot \sum_{g \in G} \overline{\sigma(g)} \widehat{\alpha}_{\tau} (f_g) 
=
\frac{1}{1+ \tfrac{\ep}{8}} \sum_{g \in G} \overline{\sigma(g)} f_g
= 
t_{\sigma}.
\end{align*}
We prove (\ref{A.7}).  For all $a \in F_1$,  we use (\ref{W.T.R.P.66.a}) and (\ref{W.T.R.P_11.b}) at the second step to estimate
\begin{align*}
\left\| 
faf - \sum_{g \in G}   f_g a f_g 
\right\|
&\leq
\left\| 
 \sum_{g, h \in G}   f_g a f_h
 -
  \sum_{g, h \in G}   f_g f_h a
\right\|
\\\notag
&\qquad+
\left\| 
  \sum_{g, h \in S}   f_g f_h a
\right\|
+
\left\| 
  \sum_{g\in G}   f^2_g a
  -
  \sum_{g\in G}   f_g a f_g
\right\| 
\\\notag
&<
\card(G) \ep' + \card(G)^2 \ep' + \card(G) \ep'
<3 \card(G)^2 \ep'.
\end{align*}
Therefore, using this, (\ref{W.T.R.P_11.b}), and the fact that 
$\widehat{\alpha}_{\tau}(faf)=faf$ for $a \in F_1$
 at the second step,
\begin{align*}
\| s_{\tau} a s^*_{\tau}
-
\widehat{\alpha}_{\tau} (f a f)\|
&\leq
\left\| 
\sum_{g, h \in G} \overline{\tau (g)} \tau (h) f_g a f_h
-
\sum_{g, h \in G} \overline{\tau (g)} \tau (h) f_g  f_h a
\right\|
\\\notag
&\qquad+
\left\|
\sum_{g, h \notin S} \overline{\tau (g)} \tau (h) f_g  f_h a
\right\|
+
\left\| 
\sum_{g \in G} |\tau (g)|^{2}  f^2_g a
-
\sum_{g \in G}   f_g a f_g
\right\|
\\\notag
&\qquad+
\left\| 
\sum_{g \in G}   f_g a f_g
- f a f
\right\|
\\\notag
&<
\card(G)^2 \ep'+ \card(G)^2 \ep'+ \card(G) \ep'+ 3 \card(G)^2 \ep'
\\\notag
&\leq 6 \card(G)^2 \ep'.
\end{align*}
This relation implies that, for all $a \in F_1$,
\begin{equation}
\label{EQ8.20210.3}
\| t_{\tau} a t^*_{\tau}
-
\widehat{\alpha}_{\tau} (c a c)\|
<
\frac{6 \card(G)^2 \ep'}{(1+ \tfrac{\ep}{8})^2}
<
\ep.
\end{equation}
Now, we use  (\ref{W.T.R.P_22.c}) at the second step to get
 \begin{align}
 \label{Eq4.20201003}
s_{\tau} u_g s^*_{\tau} 
 = 
 \sum_{h, t \in G} \overline{\tau (h)} \tau (t) f_h \alpha_{g} (f_t) u_g
 =
 \sum_{h, t  \in G} \overline{\tau (h)} \tau (g^{-1}t) f_h f_{t} u_g,
 \end{align}
 and use (\ref{W.T.R.P.55.f}) to get
 \begin{equation}
 \label{Eq5.20201003}
 \widehat{\alpha}_{\tau} (f u_g f) 
 = 
 \overline{\tau(g)} f u_g f
 =
  \overline{\tau(g)} f^2 u_g.
 \end{equation}
Therefore, using (\ref{Eq4.20201003}) and (\ref{Eq5.20201003}) at the first step,  
and using  (\ref{W.T.R.P.66.a}) and (\ref{E1.20201002}) at the second step,
\begin{align*}
\| s_{\tau} u_g s^*_{\tau}
-
\widehat{\alpha}_{\tau} (f u_g f)\|
&\leq
\left\|
  s_{\tau} u_g s^*_{\tau} - \sum_{h, t  \in G} \overline{\tau (h)} \tau (g^{-1}t) f_h f_{t} u_g
  \right\|
  \\&\quad+
\left\| 
\sum_{h, t  \notin S} \overline{\tau (h)} \tau (g^{-1}t) f_h f_{t} u_g
\right\|
+
\left\|
\sum_{h\in G} \tau (g^{-1}) f^2_h u_g - \overline{\tau(g)} f^2 u_g
\right\|
\\\notag
&<
2\card(G)^2 \ep'+ \card(G) \ep'< 3 \card(G)^2 \ep'.
\end{align*}
This relation implies that
\begin{equation}
\label{Eq7.20201003}
\| t_{\tau} u_g t^*_{\tau}
-
\widehat{\alpha}_{\tau} (c u_g c)\|
<
\frac{3\card(G)^2 \ep'}{(1+ \tfrac{\ep}{8})^2}
<\ep.
\end{equation}
Part (\ref{A.7}) now follows from (\ref{EQ8.20210.3}) and (\ref{Eq7.20201003}).

To prove (\ref{A.8}), we use (\ref{W.T.R.P.55.f}) at the third step to estimate
\begin{align*}
\|(1 - c) - (1 - f -\ep')_+ \| 
&<
\| (1 - c) + (1-f) \| + \| (1-f) - (1 - f -\ep')_+ \|
\\&\leq 
\| c - f\| + \ep'
\leq
\left|
\frac{1}{1+ \tfrac{\ep}{8}} - 1 
\right|
\cdot \| f \|+\ep'
<
\frac{\ep}{2} + \frac{\ep}{2}=\ep.
\end{align*}
Using this and Lemma~\ref{PhiB.Lem_18_4_10}(\ref{Item_9420_LgSb_1_6}) at the first step,  and using (\ref{W.T.R.P.33.d}) at second step, we get
\[
(1 - c -\ep)_+ \precsim_A (1 - f -\ep')_+ \precsim_A y \sim_{A} zxz^*.
\]
Since the above relation also holds in $C^*(G, A, \alpha)$, it follows that
\[
(1 - c -\ep)_+ \precsim_{C^*(G, A, \alpha)}  zxz^* \precsim_{C^*(G, A, \alpha)} x.
\] 

To prove (\ref{A.9}), we use the fact that $\| c \|\leq \frac{1}{1+ \tfrac{\ep}{8}}$ at the second step,  and use the second part of (\ref{Eq.b.2021.10.11}) and
 (\ref{Eq5.20201011}) at the third step to estimate
\begin{align*}
\| cz x z^*c - z cxc z^* \|
&\leq
\| cz - zc \| \cdot \| xz^*c \| + \| zcx \|\cdot \|z^*c - cz^* \|
\\\notag
&<
\sum_{g \in G}
\left\| 
fE_{g} (z) - E_{g} (z) f
\right\|
\cdot \| xz^*c \| \cdot \frac{1}{1+\tfrac{\ep}{8}}
\\&\qquad+
\frac{1}{1+\tfrac{\ep}{8}}\cdot \| zcx \| \cdot
\sum_{g \in G} 
\left\| 
E_{g} (z^*)f  - f E_{g} (z^*)
\right\|
\\\notag
&<
 \card (G)^2 \ep' + \card (G)^2 \ep'
<
\frac{\ep}{2}. 
\end{align*}
This relation implies that
\begin{equation}
\label{Eq3.20201011}
\| zcxcz^* \| > \| czxz^*c  \| - \frac{\ep}{2}.
\end{equation}
Now, using the first part of  (\ref{Eq.b.2021.10.11}) at the first step, using (\ref{Eq3.20201011}) at the second step, 
 using (\ref{W.T.R.P.44.e}) at the fourth step, and  using the third part of  (\ref{Eq.b.2021.10.11}) at the fifth step, we get 
\begin{align*}
\| cxc\| \geq  \| z cxc z^*\| 
> \| czxz^*c  \| - \frac{\ep}{2}
&=
\frac{1}{(1+\tfrac{\ep}{8})^2} \cdot \| fzxz^*f  \| - \frac{\ep}{2}
\\&>
 \frac{1 - \ep'}{(1+\tfrac{\ep}{8})^2} \cdot \| z x z^*\| - \frac{\ep}{2}
\\& >
\frac{(1 - \tfrac{\ep}{8})^2}{(1+\tfrac{\ep}{8})^2} -  \frac{\ep}{2}
 >
 1 - \frac{\ep}{2} - \frac{\ep}{2} = 1 -\ep.
\end{align*}
This completes the proof.
\end{proof}
\begin{lem}
\label{Lem_Cunt_a_n}
Let $A$ be a unital C*-algebra and let $\ep>0$. 
Let $a \in A_+$ with $\|a \|\leq 1$ and let $n \in \N$. Then
$(1 - a^n - \ep )_+ \precsim_A \left(1 - a - \tfrac{\ep}{n} \right)_+$.
\end{lem}
\begin{proof}
We apply  functional calculus to $a$.
 So we have an isomorphism
 $\varphi \colon C\left( \spec (a) \right) \to  C^*(1, a)$.
 Then we can check  that
 \[
 \big\{ x \in \spec (a) \colon 1 - x^n - \ep > 0 \big\} \subseteq
 \big\{ x \in \spec (a) \colon  1 - x - \tfrac{\ep}{n}  > 0 \big\}.
 \]
This completes the proof.
 \end{proof}
\begin{prp}\label{ProWTAR->WTRP}
Let $A$ be an infinite-dimensional simple unital C*-algebra and let $G$ be a finite abelian group.
Let $\alpha \colon G \to \Aut (A)$ be an action of  $G$ on $A$ which is pointwise outer.
If $\alpha$ is weakly tracially approximately representable, then $\widehat{\alpha}$ has the weak tracial Rokhlin property.
\end{prp}
\begin{proof}
Let $\ep\in (0, 1)$, let $x \in C^*(G, A, \alpha)_+$ with $\| x \|=1$, and let $F \subset C^*(G, A, \alpha)$ be a finite set.
 Without loss of generality,  
 $F= F_0 \cup \{u_g \colon g \in G\}$ for some finite subset $F_0 \subset A$ such that 
 $\| a \| \leq 1$ for all $a \in F_0$.
By Lemma~\ref{L_0X19_Norm_fa_x_fa}, we choose $\dt\in (0, \ep)$ such that:
\begin{enumerate}
\item\label{Cond.a.05.15.19}
If $a, b \in A_{+}$ satisfy
$\| a \| \leq 1$, $\| b \| \leq 1$, and
$\| a b a \| > 1 - \dt$,
then 
\[
\| a^4 b a^4 \| > 1 - \frac{\ep}{8}.
\]
\setcounter{TmpEnumi}{\value{enumi}}
\end{enumerate}
Since $\alpha$ is pointwise outer, it follows from Lemma~\ref{L_0X12_ModPasn} that there is an element
 $z \in C^*(G, A, \alpha) \setminus \{0\}$ such that:
\begin{equation}
\label{Eqa.2021.10.11}
zxz^* \in A_+,
\qquad
\| z\| \leq 1,
\qquad
\mbox{and}
\qquad
\|zxz^* \| >1 - \frac{\ep}{8}.
\end{equation}
Set
\[
F_0 = \Big\{E_g (a), E_g (a)^{*}  \colon a \in F\cup \{z, x\}, \ g \in G\Big\},
\]
\[
y= \frac{1}{\| z x z^*\|} \cdot z x z^*,
\quad
\mbox{ and } 
\quad
\ep'= \frac{\dt}{640\card (G)}.
\]
Using Definition~\ref{Week.Ap.Rep.Def} with $F_0 $, $y$, and $\ep'$ as given,
there are
 $c  \in (A^{\alpha})_+$ with $\| c \| \leq 1$ and 
 contractive elements  $s_g \in  \overline{c A c}$ 
 for $g \in G$ such that:
\begin{enumerate}
\setcounter{enumi}{\value{TmpEnumi}}
\item
\label{Week.Ap.Rep.Def.11}
$\| c a - a c \| < \ep'$ for all $a \in F_0 \cup \{s_g \colon g \in G\}$.
\item
\label{Week.Ap.Rep.Def.22}
$\| \alpha_g (c a c) - s_g a  s^*_g \|< \ep'$ for all $a \in F_0$ and all $g \in G$.
\item
\label{Week.Ap.Rep.Def.33}
$\| s_g s_h - c s_{gh}\| < \ep'$ for all $g, h \in G$.
\item
\label{Week.Ap.Rep.Def.33"}
$\| s_1 - c\|< \ep'$ and $\|s^*_g - s_{g^{-1}}  \|< \ep'$ for all $g \in G$.
\item
\label{Week.Ap.Rep.Def.44}
$\| \alpha_g (s_h) - s_h\| < \ep'$ for all $g, h \in G$.
\item
\label{Week.Ap.Rep.Def.55}
$(1 - c - \ep')_+ \precsim_A y$.
\item
\label{Week.Ap.Rep.Def.66}
$\| c y c\| > 1-\ep'$.
\setcounter{TmpEnumi}{\value{enumi}}
\end{enumerate}
We use (\ref{Week.Ap.Rep.Def.66}) and (\ref{Cond.a.05.15.19}) to get 
\begin{equation}
\label{Eq1_20201116}
\| c^4 y c^4 \| > 1 - \frac{\ep}{8},
\end{equation}
and use (\ref{Week.Ap.Rep.Def.44}) at the second step to get, for all $g, h \in G$,
\begin{equation}
\label{Eq2.20201105}
\| u_g s_h - s_h u_g\| 
\leq 
\| \alpha_{g} (s_h) - s_h \| \cdot \|u_g \|
<
\ep'.
\end{equation}
Using the fact that $s^*_{hg}=  s^*_{gh}$ at the first step, and using (\ref{Week.Ap.Rep.Def.33}) and  (\ref{Eq2.20201105})  at the second step, we estimate
\begin{align}
\label{Eq4_20201105}
&\|u_g s^*_g u_h s^*_h -  u_{g h} s^*_{g h} c\|
\\\notag 
&\hspace*{2em}\leq
\| u_g \| \cdot \| s^*_g u_h - u_h s^*_g \| \cdot \|s^*_h \|
+
\| u_{gh}\|\cdot \| s^*_g s^*_h -  s^*_{gh} c\|
<
\ep'+\ep'=2\ep'.
\end{align}
For every $\tau \in \widehat{G}$,  we define 
\[
p_{\tau} = \frac{1}{\card (G)} \sum_{g \in G} \overline{\tau(g)}  u_g s^*_g,
\qquad
f_{\tau} = c p^*_{\tau} p_{\tau} c, 
\qquad
\mbox{ and } 
\qquad
f=\sum_{\tau \in \widehat{G}} f_{\tau}.
\]
Now,  we claim that:
\begin{enumerate}
\setcounter{enumi}{\value{TmpEnumi}}
\item
\label{WTR_a}
$f_{\tau}$ for $\tau \in \widehat{G}$ are positive contractions.
\item
\label{WTR_b}
$\| f_\tau f_\sigma \| < \ep$
for all $\tau, \sigma \in \widehat{G}$ with $\tau \neq \sigma$.
\item
\label{WTR_c}
$\| a f_\tau - f_\tau a \| <\ep$
for all $\tau \in \widehat{G}$ and all $a \in F$.
\item
\label{WTR_d}
 $\| \alpha_{\tau} ( f_\sigma) - f_{\tau \sigma}\| < \ep$ for all $\tau, \sigma \in \widehat{G}$.
\item
\label{WTR_e}
$(1 - f - \ep)_{+} \precsim_{C^*(G, A, \alpha)} x$.
\item
\label{WTR_f}
$\| f x f \| > 1 - \ep$.
\setcounter{TmpEnumi}{\value{enumi}}
\end{enumerate}
Part (\ref{WTR_a}) is immediate.

We prove (\ref{WTR_b}). We use (\ref{Eq2.20201105}) at the last step to estimate
\begin{align}
\label{EQ.3.20200815}
\|p_{\tau} c - c p_{\tau} \|
&\leq
\frac{1}{\card(G)} \sum_{g \in G} |\tau(g)| \cdot \| u_g s^*_g c - c u_g s^*_g\|
\\\notag &\leq
\frac{1}{\card(G)} \sum_{g \in G}
\left(
\|u_g \| \cdot \| s^*_g c - c s^*_g\|
+
\| c \| \cdot \| u_g s^*_{g} - s^*_{g} u_g \|
\right)
\\\notag & <
2\ep'.
\end{align}
Using (\ref{EQ.3.20200815}), short computations show that
\begin{equation}
\label{EQ1_20201107}
\| p_{\tau} c^2 - c^2 p_{\tau} \|< 4\ep',
\quad
\| p^*_{\tau} p_{\tau} c  -  c  p^*_{\tau} p_{\tau} \|<4\ep',
\quad
\mbox{and}
\quad
\| p_{\tau} c^3  -  c^3  p_{\tau} \|<6\ep'.
\end{equation}
 We use  (\ref{Eq2.20201105}) at the last step  to get
\begin{align}
\label{EQ.1.20201105}
\| p^*_{\tau} - p_{\tau} \|
&\leq
\frac{1}{\card (G)} \sum_{g \in G}
|\tau(g)|\cdot
\| 
s_{g^{-1}} u_{g}  - u_{g} s^*_{g}
\|
\\\notag
&\leq
\frac{1}{\card (G)} \sum_{g \in G}
\left(
\| 
s_{g^{-1}} u_{g}  -  s^*_{g} u_{g}
\|
+
\| 
s^*_{g} u_{g} - u_{g} s^*_{g}
\|
\right)
\\\notag &<
\ep'+\ep'=
2\ep'.
\end{align}
A short computation together with Remark~\ref{rmk211005} shows that
\begin{equation}
\label{EQ13_20201106}
 \sum_{g, h \in G} \overline{\sigma(g)} \overline{\tau(h)} u_{g h} s^*_{g h} 
 =
 \begin{cases}
0                  &  \sigma \neq \tau\\
\card (G) p_{\tau} & \sigma = \tau.
 \end{cases}
\end{equation}
Then,  by (\ref{Eq4_20201105}),
\begin{align}
\label{EQ.2.20200815}
&\left\|
p_{\sigma} p_{\tau}
-
\frac{1}{\card(G)^2} \sum_{g, h \in G} \overline{\sigma(g) \tau(h)} u_{g h} s^*_{g h} c
\right\|
\\
\notag &\hspace*{5em}\leq
\frac{1}{\card(G)^2}
\sum_{g, h \in G} | \sigma(g) \tau(h)|
\cdot 
\| u_g s^*_g u_h s^*_h - u_{g h} s^*_{g h} c\|
<
2\ep'.
\end{align}
Using (\ref{EQ13_20201106}) and (\ref{EQ.2.20200815}), we get, for all $\tau, \sigma \in \widehat{G}$ with $\tau\neq \sigma$,
\begin{equation}\label{Eq1.2021.10.08}
\|p_{\tau} p_{\sigma} \|< 2\ep',
\end{equation}
and, for all $\tau, \sigma \in \widehat{G}$,
\begin{equation}\label{Eq2.2021.10.08}
\| p^2_{\tau} -  p_{\tau} c\| < 2\ep'.
\end{equation}
Therefore, for all $\tau, \sigma \in \widehat{G}$ with $\tau \neq \sigma$, using the first part of (\ref{EQ1_20201107}),  (\ref{EQ.1.20201105}), and (\ref{Eq1.2021.10.08}) at the last step,
\begin{align}\label{Eq.35.2021.10.08}
\| f_{\tau} f_{\sigma}\| = \| c p^*_{\tau} p_{\tau} c^2 p^*_{\sigma} p_{\sigma} c\|
&\leq 
\|c p^*_{\tau} \| \cdot \|p_{\tau} c^2 - c^2 p_{\tau}  \| \cdot \| p^*_{\sigma} p_{\sigma} c\|
\\\notag&\quad+
\| c p^*_{\tau} c^2 p_{\tau} \| \cdot \|p^*_{\sigma} - p_{\sigma} \| \cdot \| p_{\sigma} c\|
\\\notag&\quad+
\| c p^*_{\tau} c^2\| \cdot \|p_{\tau} p_{\sigma} \| \cdot \| p_{\sigma} c\|
\\\notag &<
4\ep' + 2\ep' + 2\ep'= 8\ep'. 
\end{align}

We prove (\ref{WTR_c}).
 For all $v \in F\cup \{z, x\}$,  by (\ref{Week.Ap.Rep.Def.11}),  we have
\begin{equation}\label{EQ45_20201107}
\left\| 
c v - vc
\right\|
\leq
\sum_{g\in G} \| c E_{g} (v) - E_{g} (v) c  \|
< \card(G) \ep'.
\end{equation}
For all $a\in F_0$, we use (\ref{Eq2.20201105}),  (\ref{Week.Ap.Rep.Def.11}),  
(\ref{Week.Ap.Rep.Def.22}), (\ref{Week.Ap.Rep.Def.33}), and (\ref{Week.Ap.Rep.Def.44}) at the last step to get
\begin{align*}
\| cac (u^*_g s_g) -(u^*_g s_g) cac \|
&\leq
\| cac \| \cdot \| u^*_g s_g - s_g u^*_g \| 
\\&\quad +
\|u^*_g \| \cdot \| \alpha_{g} (cac) - s_g a s^*_g \| \cdot \| \alpha_{g} ( s_g)\|
\\&\quad+
\|u^*_g s_g a s^*_g\| \cdot \| \alpha_{g} ( s_g) - s_g \|
\\&\quad+
\|u^*_g s_g a \| \cdot \| s^*_g - s_{g^{-1}} \| \cdot \| s_g\|
+
\|u^*_g s_g a \| \cdot \|s_{g^{-1}} s_g - c s_1 \|
\\&\quad +
\|u^*_g s_g a c \| \cdot \|s_1 - c \|
 +
\|u^*_g s_g \| \cdot \|ac - ca \| \cdot \| c\|
\\&<
7\ep'.
\end{align*}
This relation implies that
$\| cac p_{\tau} - p_{\tau} cac\|< 7\ep'$
for all $a\in F_0$.
Using this and the fact that $\| p_{\tau}\| \leq 1$, we get 
\begin{equation}\label{EQ3.2021.10.08}
\| cac p^*_{\tau} p_{\tau} - p^*_{\tau} p_{\tau} cac  \|
\leq
\| cac p^*_{\tau} -  p^*_{\tau} cac\| \cdot \| p_{\tau}\|
+
\| p^*_{\tau} \| \cdot \| cac p_{\tau} -   p_{\tau} cac \| < 14\ep'.
\end{equation}
For all $h \in G$ and all $\tau \in \widehat{G}$,  we use (\ref{Eq2.20201105}) to get
\begin{align}\label{EQ.6.20200815}
\| u_h p_{\tau}  -  p_{\tau} u_h \|
\leq
\frac{1}{\card(G)} \sum_{g \in G} |\tau(g)| \cdot \|  u_h u_g s^*_g -  u_g s^*_g u_h\|
< \ep'.
\end{align}
 For all $a \in F_0$,  we use (\ref{EQ3.2021.10.08})  and 
 the second part of (\ref{EQ1_20201107}) at the last step to get
\begin{align}\label{EQ.d.2021.10.08}
\| a f_{\tau} - f_{\tau} a \|
&\leq
\| a c - c a\| \cdot \| p^*_{\tau} p_{\tau} c \| 
+
\| c a\| \cdot \|  p^*_{\tau} p_{\tau} c -  c p^*_{\tau} p_{\tau}  \| 
\\\notag&\quad+
 \| cac p^*_{\tau} p_{\tau}  -   p^*_{\tau} p_{\tau} cac  \| 
+
\|  p^*_{\tau} p_{\tau} c -  c p^*_{\tau} p_{\tau}  \| \cdot \| ac\|
\\\notag&\quad+
\| c p^*_{\tau} p_{\tau} \| \cdot \| ac  -  ca \| 
\\\notag&<
 \ep' + 4\ep'+
14\ep' + 4 \ep'
+\ep'
=24\ep'.
\end{align}
For all $h \in G$, by (\ref{EQ.6.20200815}), we have
\begin{equation}\label{EQ.d.2021.10.08'}
\| u_h f_{\tau} - f_{\tau} u_h\| \leq
\| c \| \cdot \| u_h p^*_{\tau} -  p^*_{\tau} u_h\| \cdot \|p_{\tau} c \|
+
\|c p^*_{\tau} \| \cdot \| p_{\tau}  u_h -   u_h p_{\tau} \| \cdot \| c\| < 2\ep'.
\end{equation}
Now putting (\ref{EQ.d.2021.10.08}) and (\ref{EQ.d.2021.10.08'}) together, we get, for all $v \in F \cup \{x, z\}$,
\begin{equation}\label{Eq30.2021.10.09}
\| v f_{\tau} -  v f_{\tau}\| < 26\card (G)\ep'<\ep.
\end{equation}

We prove (\ref{WTR_d}).
We use (\ref{Eq2.20201105}) at the last step to estimate
\begin{align}
\label{Eq10_20201106}
\| \widehat{\alpha}_{\sigma} (u_g s^*_g ) - \overline{\sigma(g)} u_g  s^*_g   \|
&\leq
\|\widehat{\alpha}_{\sigma} \| \cdot \| u_g s^*_g - s^*_g u_g  \|
+
| \sigma(g)| \cdot \| s^*_g u_g  - u_g  s^*_g \|
\\\notag
&<
\ep'+\ep'=2\ep'.
\end{align}
Now, using (\ref{Eq10_20201106}) at the last step, we get
\begin{align}\label{Eq10_20201106.b}
\| \widehat{\alpha}_{\sigma} (p_{\tau})
-
p_{\sigma \tau}
\|
=
\frac{1}{\card(G)} \sum_{g \in G} |\tau(g)| \cdot \| \widehat{\alpha}_{\sigma} (u_g s^*_g ) - \sigma(g) u_g  s^*_g  \|
<
2\ep'.
\end{align}
Therefore, using (\ref{Eq10_20201106.b}) at the last step,
\begin{align*}
\| \widehat{\alpha}_{\tau} ( f_\sigma) - f_{\tau \sigma} \| 
&\leq
\|c \| \cdot \| \widehat{\alpha}_{\tau}  ( p^*_\sigma  ) -  p^*_{\tau \sigma} \| \cdot \| \widehat{\alpha}_{\tau}  ( p_\sigma  ) c\|
\\&\quad+
\|c p^*_{\tau \sigma}\| \cdot \| \widehat{\alpha}_{\tau}  ( p_\sigma  ) -  p_{\tau \sigma} \| \cdot \| c\|
\\&<
2 \ep' + 2 \ep'< \ep.
\end{align*}

We prove (\ref{WTR_e}). 
 We use Remark~\ref{rmk211005} at the first step
 and use (\ref{Week.Ap.Rep.Def.33"}) at the second step to get
\begin{equation}\label{EqK.2021.10.10}
\left\|
\sum_{\tau \in \widehat{G}} p_{\tau} - c
\right\|
=
\|u_1 s^*_1 - c \|
< \ep'.
\end{equation}
 Using (\ref{EQ.1.20201105}), (\ref{Eq2.2021.10.08}), and (\ref{EqK.2021.10.10}) at second step,
 we get
\begin{align*}
\|f - c^4 \| 
&\leq
\sum_{\tau \in \widehat{G}} \|c p_{\tau} \| \cdot \|   p^*_{\tau} - p_{\tau} \| \cdot \|c \|
+
\sum_{\tau \in \widehat{G}} \|c \| \cdot \| p^2_{\tau}   - p_{\tau} c\| \cdot \|c\| 
\\&\qquad+
\|c \|\cdot
\left\|
\sum_{\tau \in \widehat{G}}  p_{\tau} 
-
c
\right\|
\cdot \| c^2\|
\\&<
2\card(G) \ep'+ 2\card(G) \ep' + \ep' \leq 5\card(G) \ep'.
\end{align*}
Using this at the second step, we get
\begin{equation}
\label{Eq2_20201107}
\|(1 - f) - (1 -c^4) \| = \| f - c^4 \|<5\card(G) \ep'< \frac{\dt}{2}.
\end{equation}
Using (\ref{Eq2_20201107}) and Lemma~\ref{PhiB.Lem_18_4_10}(\ref{Item_9420_LgSb_1_6}) at the second step,   using Lemma~\ref{Lem_Cunt_a_n} at the third step,  
and using (\ref{Week.Ap.Rep.Def.55}) at fifth step, we get 
\begin{align*}
(1 - f - \ep)_+ 
\precsim_{C^*(G, A, \alpha)}
\left(1 - f - (\tfrac{\dt}{2} +\tfrac{\dt}{2})\right)_+ 
&\precsim_{C^*(G, A, \alpha)} 
\left(1 -c^4 - \tfrac{\dt}{2}) \right)_+ 
\\&\precsim_{C^*(G, A, \alpha)}
 \left(1 -c - \tfrac{\dt}{8}\right)_+
 \\&\precsim_{C^*(G, A, \alpha)}
 (1 -c - \ep')_+
\\&\precsim_{C^*(G, A, \alpha)} y
\precsim_{C^*(G, A, \alpha)}
x.
\end{align*}

To prove (\ref{WTR_f}),  we use the third part of (\ref{EQ1_20201107}),   (\ref{EQ.1.20201105}),  (\ref{Eq2.2021.10.08}), and  (\ref{EQ.3.20200815})  at the last step to estimate
\begin{align}\label{EQ.b.20211004}
&\|c p_{\tau} p^*_{\tau} c^2 p_{\tau} p^*_{\tau} c - c^4 p_{\tau} c^3  \|
\\\notag& \hspace*{4em} \leq
\|c p_{\tau} \| \cdot
\|  p^*_{\tau}  -  p_{\tau}\|
\cdot \| c^2 p_{\tau} p^*_{\tau} c\|
 +
\|c \| \cdot
\|
 p^2_{\tau}   -   p_{\tau} c
\|
\cdot \| c^2 p_{\tau} p^*_{\tau} c\|
\\\notag
&\hspace*{6em} +
\|c \| \cdot
\|
  p_{\tau} c^3  -  c^3  p_{\tau} 
\|
\cdot \| p_{\tau} p^*_{\tau} c\| 
+
\| c^4 \| \cdot
\|
  p^2_{\tau}  -    p_{\tau} c 
\|
\cdot \| p^*_{\tau} c \|
\\\notag
&\hspace*{6em} +
\|c^4   \| \cdot 
\|
 p^2_{\tau}   -   p_{\tau} c 
\|
\cdot
\| p^*_{\tau} c \|
+
\| c^4  p_{\tau} \| \cdot
\|
 c  p^*_{\tau}  -   p^*_{\tau}  c
\|
\cdot \| c\|
\\\notag
&\hspace*{6em} +
\| c^4  p_{\tau} \| \cdot
\|
 p^*_{\tau} -  p_{\tau}
\|
\cdot 
\|c^2 \|
+
\|c^4 \| \cdot 
\| p^2_{\tau} - p_{\tau} c\| 
\cdot \| c^2\|
\\\notag &\hspace*{4em} <
2\ep' + 2 \ep' + 6\ep'+2\ep'+ 2\ep'+ 2\ep'+ 2\ep'+ 2\ep'
=
20\ep'. 
\end{align}
Using (\ref{EQ45_20201107}), (\ref{EqK.2021.10.10}), and (\ref{EQ.b.20211004})  at the last step,  we get
\begin{align*}
\left\|
 \sum_{\tau \in \widehat{G}} f^2_{\tau} x - c^4 x c^4
 \right\|
 &\leq
 \left\|
 \sum_{\tau \in \widehat{G}} c p_{\tau} p^*_{\tau} c^2 p_{\tau} p^*_{\tau} c x
  -
 \sum_{\tau \in \widehat{G}}  c^4 p_{\tau} c^3 x
 \right\|
 \\&\quad+
\| c^4\| \cdot
  \left\| \sum_{\tau \in \widehat{G}}  p_{\tau} - c \right\| 
\cdot \|c^3 x\|
 +
 \| c^8 x - c^4 x c^4\|
 \\&<
20 \card(G) \ep'+ \ep' + 4 \card(G) \ep'\leq 25 \card(G) \ep'.
\end{align*}
Now, using this, (\ref{Eq30.2021.10.09}), and (\ref{Eq.35.2021.10.08}) at the second step,  we get
\begin{align}\label{EQ.a.2021.10.05}
 \| f x f - c^4 x c^4\|
&\leq
 \left\|
  \sum_{\tau, \sigma \in \widehat{G}} f_{\tau} x f_{\sigma}  
 - 
 \sum_{\tau, \sigma \in \widehat{G}} f_{\tau} f_{\sigma} x 
 \right\|
+
\left\|
  \sum_{\tau\neq \sigma } f_{\tau} f_{\sigma} x   
 \right\|
 \\ \notag &\quad+
\left\|
 \sum_{\tau, \sigma \in \widehat{G}} f^2_{\tau} x - c^4 x c^4
 \right\|
\\ \notag &<
26 \card(G)^3 \ep'
+
8\card(G)^2 \ep'
+
25 \card(G) \ep'< \frac{\ep}{4}.
\end{align}
We further use (\ref{EQ45_20201107}) and the second part of  (\ref{Eqa.2021.10.11}) to get
\begin{align}
\label{EQ12021.10.13}
\|z c^4 x c^4 z^* - c^4 z x z^* c^4\|
&<
\|z c^4 -  c^4 z \| \cdot \| x c^4 z^* \|
+
\|c^4 z x \| \cdot \|c^4 z^*  - z^* c^4 \|
\\\notag
&<
4 \card (G) \ep' + 4 \card (G) \ep'< \frac{\ep}{8}.
\end{align}
Now, using (\ref{EQ.a.2021.10.05}) at the first step, using the second part of (\ref{Eqa.2021.10.11}) at the second step, 
using (\ref{EQ12021.10.13}) at the third step, 
 using (\ref{Eq1_20201116}) at the fourth step,  and using the third part of  (\ref{Eqa.2021.10.11}) at the fifth step,  we get
\begin{align*}
\| f x f\| \geq \|c^4 x c^4 \| - \frac{\ep}{8}
\geq
\|z c^4 x c^4 z^* \| - \frac{\ep}{8}
&>
\| c^4 z x z^* c^4  \| - \frac{\ep}{8} -\frac{\ep}{8}
\\&>
\| z x z^* \| \cdot \left(1 - \frac{\ep}{8}\right) - \frac{\ep}{4}
\\&>
\left(1 - \frac{\ep}{8}\right)^2  - \frac{\ep}{4}
> 1-\ep.
\end{align*}
This completes the proof of the claim and the result follows from the claim and Lemma~\ref{invariant_contractions}.
\end{proof}
Now, putting Proposition~\ref{ProWTRP->WTAR} and Proposition~\ref{ProWTAR->WTRP} together, we get the following theorem.
\begin{thm}
\label{MainThmCoaction}
Let $A$ be an infinite-dimensional simple unital C*-algebra and let $G$ be a finite abelian group.
Let $\alpha \colon G \to \Aut (A)$ be an action of  $G$ on $A$ which is pointwise outer.
Then:
\begin{enumerate}
\item \label{MainThmCoaction1}
$\alpha$ has the weak tracial Rokhlin property if and only if $\widehat{\alpha}$ is weakly tracially approximately representable.
\item \label{MainThmCoaction2}
$\alpha$ is weakly tracially approximately representable if and only if $\widehat{\alpha}$ has the weak tracial Rokhlin property.
\end{enumerate}
\end{thm}
It is an important question whether there exists a finite group action 
which is weakly tracially approximately representable but not tracially approximately representable.
To obtain such examples, it suffices to find an action of a finite abelian group which
has the weak tracial Rokhlin property but not the tracial Rokhlin property by considering the duality theorems for tracial Rokhlin
actions and weak tracial Rokhlin actions (Theorem~3.11 of \cite{Ph11} and Theorem~\ref{MainThmCoaction}).
One such example was constructed in Example 5.10 of \cite{HO13}.
Furthermore, by  tensoring this action with  the trivial action of $G$ on $A$ and using Theorem~\ref{MainThmCoaction}, we can produce many more examples of group actions which  are weakly tracially approximately representable but not tracially approximately representable. 

Here is an example of a finite non-abelian group action on a unital simple 
$\mathcal{Z}$-stable
C*-algebra which simultaneously  is weakly tracially approximately representable and has the weak tracial Rokhlin property. 
\begin{exa}\label{Exam1}
Let $A$ be the $2^\infty$ UHF algebra, let $m \in \N$, and let $S_m$ denote the group of all permutation of 
$\{1, 2, \ldots, m\}$. Let 
$\alpha \colon S_m \to \Aut (A^{\otimes m})$ 
given by
\[
\alpha_{\sigma} (a_1 \otimes a_2 \otimes \ldots \otimes a_m)
=
a_{\sigma^{-1}(1)} \otimes a_{\sigma^{-1}(2)} \otimes \ldots \otimes a_{\sigma^{-1}(m)}
\]
for $\sigma \in S_m$ and $a_1, a_2, \ldots, a_m \in A$. 
It follows from Lemma~3.13 of \cite{AV20} that $\alpha$ is approximately representable and therefore it is weakly tracially approximately representable. 
It follows from  Theorem~3.10 of \cite{AGJP22} that
 $\alpha$ has the weak tracial Rokhlin property.
\end{exa}

 Here is an example  of a finite abelian group action on a unital simple $\mathcal{Z}$-stable C*-algebra which 
simultaneously  is weakly tracially approximately representable and has the weak tracial Rokhlin property, 
 but is neither  tracially approximately representable nor has the  tracial Rokhlin property. 
\begin{exa}\label{Exam2}
Let $\mathcal{Z}$ be the Jiang–Su algebra.
Let $\alpha \colon
 \mathbb{Z}/2\mathbb{Z} \to \Aut (\mathcal{Z} \otimes \mathcal{Z})$ be the flip action. Because $\alpha$ has the weak tracial Rokklin property (see example Example 5.10 of \cite{HO13}),
 $C^*(\mathbb{Z}/2\mathbb{Z}, \mathcal{Z} \otimes \mathcal{Z}, \alpha)$ is simple and has a unique trace. 
 Since  the crossed product von Neumann
algebra is a factor, it follows that the dual action 
$\widehat{\alpha}$ is strongly outer in the sense of Definition~2.5 of \cite{GHV22}. Then, by Corollary~8.6 of \cite{GHV22}, $\widehat{\alpha}$
has the weak tracial Rokhlin property and therefore  $\alpha$ is weakly tracially approximately representable by Theorem~\ref{MainThmCoaction} and 
it is obviously not tracially approximately representable. 
\end{exa}

Here is an example of a finite non-abelian group action on a unital simple 
non-$\mathcal{Z}$-stable
C*-algebra which is weakly tracially approximately representable, but does not have the weak tracial Rokhlin property. 
\begin{exa}
\label{EXAmp3}
Let $G$ be a finite group, let $A$ be the unital simple separable  non-$\mathcal{Z}$-stable AH algebra with stable rank one as in Construction~5.8 of \cite{AV20}, and let $\alpha \colon G \to \Aut(A)$ be the action as in Construction~5.8 of \cite{AV20}. It follows from Corollary~5.21 of \cite{AV20} that $\alpha$ is an approximately representable but pointwise outer action.
So $\alpha$ is also weakly tracially approximately representable. Moreover,
 neither  $\alpha$ nor $\alpha \otimes \mathrm{id}_{\mathcal{Z}}$ has the weak tracial Rokhlin property.
\end{exa}
Here is an example of a finite abelian group action on a unital simple 
non-$\mathcal{Z}$-stable
C*-algebra with the weak tracial Rokhlin property, 
 but
 it is not weakly tracially approximately representable.
 \begin{exa}
 \label{EXAmp4}
 Let 
 $\alpha \colon \mathbb{Z}/2 \mathbb{Z} \to \Aut(A)$ 
 be the action as in Construction~6.1 of \cite{AGP19}. Then, by Lemma~6.4 of \cite{AGP19},  $\alpha$ has the Rokhlin property. It follows from Proposition~4.13 of \cite{AV20} and Theorem~6.15 of \cite{AGP19} that $\alpha$ is not weakly tracially approximately representable. 
 \end{exa}

\end{document}